\documentclass[10pt]{amsart}

\usepackage{amsmath,amsfonts, latexsym,graphicx, amssymb, mathtools, enumerate, mdwlist, amscd}

\usepackage{hyperref}
\hypersetup{colorlinks=true, urlcolor=blue, citecolor=blue, linkcolor=blue}

\newcommand{\U}{{\mathcal U}}
\newcommand{\0}{{\mathbf 0}}
\newcommand{\C}{{\mathbb C}}
\newcommand{\Z}{{\mathbb Z}}

\newcommand{\mult}{{\operatorname{mult}}}

\newtheorem{defn0}{Definition}[section]
\newtheorem{prop0}[defn0]{Proposition}
\newtheorem{conj0}[defn0]{Conjecture}
\newtheorem{thm0}[defn0]{Theorem}
\newtheorem{lem0}[defn0]{Lemma}
\newtheorem{corollary0}[defn0]{Corollary}
\newtheorem{example0}[defn0]{Example}
\newtheorem{remark0}[defn0]{Remark}
\newtheorem{question0}[defn0]{Question}
\newtheorem{exercise0}[defn0]{Exercise}

\newenvironment{prop}{\begin{prop0}}{\end{prop0}}

\newenvironment{thm}{\begin{thm0}}{\end{thm0}}
\newenvironment{lem}{\begin{lem0}}{\end{lem0}}
\newenvironment{cor}{\begin{corollary0}}{\end{corollary0}}

\newenvironment{exm}{\begin{example0}\rm}{\end{example0}}
\newenvironment{rem}{\begin{remark0}\rm}{\end{remark0}}

\newcommand{\propref}[1]{Proposition~\ref{#1}}
\newcommand{\thmref}[1]{Theorem~\ref{#1}}
\newcommand{\lemref}[1]{Lemma~\ref{#1}}
\newcommand{\corref}[1]{Corollary~\ref{#1}}

\newcommand{\secref}[1]{Section~\ref{#1}}

\newcommand{\mbf}[1]{{\mathbf #1}}

\title{A L\^e number addendum to the Morse Splitting Lemma}



\author{David B. Massey}

\date{}

\begin{document}

\begin{abstract} We investigate how the complex analytic Morse Splitting Lemma relates to the generic L\^e numbers of a function, and what the relation tells us about the function. As special cases, we consider functions for which the 0-dimensional generic L\^e number is at most 3 and functions for which the critical locus has dimension 1.
\end{abstract}

\maketitle

\section{Introduction}

 \medskip
 
 Let $\U$ be a connected open neighborhood of the origin in $\C^{n+1}$, let $\mbf z:=(z_0, \dots, z_n)$ be coordinates on $\U$,  and let $f:(\U, \0)\rightarrow (\C, 0)$ be a non-constant complex analytic function. After possibly choosing $\U$ smaller, we may assume that the critical locus, $\Sigma f$, of $f$ is contained in the hypersurface $V(f):=f^{-1}(0)$. We assume that $\0\in\Sigma f$ and let $s:=\dim_\0\Sigma f$.
 
 Since 1990, we have written many, many articles about the L\^e numbers $\lambda^*_{f,\mbf z}(p)$ of hypersurface singularities at points $p\in\Sigma f$ with respect to a linear choice of coordinates $\mbf z$; see, for instance, \cite{levar1}, \cite{levar2},  \cite{lecycles}, \cite{lemassey}, and \cite{lemodtrace}. 
 
At each $p\in\Sigma f$, there is a generic linear choice of coordinates $\mbf z$ such that all of the L\^e numbers are defined; see Theorem 1.28 of \cite{lecycles}. However, for actually calculating in examples, we do not want to simply appeal to the existence of generic coordinates. It is a primary feature of the L\^e numbers that, given a linear choice of coordinates, one can effectively calculate if the L\^e numbers are defined and find their values for the given choice of coordinates $\mbf z$; see Definition 1.11 of \cite{lecycles}. The main results that require only that the L\^e numbers are defined involve Thom's $a_f$ condition and can be found in Chapter 6 of \cite{lecycles}.

However, there have been results where we, and others, have used that, at a given point $p\in \Sigma f$, there are coordinates $\mbf z$ such that the L\^e numbers $\lambda^*_{f, \mbf z}(p)$ take on their generic values (see, for instance, Chapter 10 of \cite{lecycles}); we refer to such coordinates and L\^e numbers as {\bf numerically generic}. Using the complex analytic Curve Selection Lemma and applying Corollary 4.16 of \cite{lecycles}, the numerically generic L\^e numbers at $p$ are those such that  the tuple of L\^e numbers 
 $$
 \big(\lambda^s_{f,\mbf z}(p), \lambda^{s-1}_{f,\mbf z}(p), \dots, \lambda^1_{f,\mbf z}(p), \lambda^0_{f, \mbf z}(p)\big)
 $$
 attains its minimal value in the lexicographical ordering. Note that it is absolutely false that one can select a single tuple of coordinates which are numerically generic at all points near $p$. Henceforth, when we discuss the generic L\^e numbers at $p$, we mean the L\^e numbers with respect to numerically generic coordinates, and we will suppress the reference to the coordinates in the notation. It is trivial that the generic L\^e numbers are invariant under analytic isomorphisms.

Without loss of generality, we take $\0$ as our point of interest in $\Sigma f$. We prove one theorem in this paper, a theorem which is just a L\^e number supplement to the Morse Splitting Lemma, Theorem 2.47 of \cite{GLS}:

\smallskip

\begin{thm}\label{thm:mainintro} Suppose that the rank of the hessian matrix of $f$ at $\0$ is $m$. If $m=n+1$, we are in the standard case where $f$ has a complex nondegenerate isolated critical point at $\0$, i.e., an isolated critical point with Milnor number 1.

If $m\leq n$, then there is an analytic change of coordinates $\eta$ at the origin in $\C^{n+1}$ and $g\in\C\{u_0, \dots, u_{n-m}\}$ such that
$$
(f\circ \eta)(w_1,\dots, w_m, u_0, \dots, u_{n-m}) = w_1^2+\dots+w_m^2+g(u_0, \dots, u_{n-m}),
$$
where $\mult_\0 g\geq 3$, $\Sigma (f\circ \eta)=\{\0\}\times\Sigma g$, and the generic L\^e numbers of $f$ equal the generic L\^e numbers of $g$, i.e., for all $k$ such that $0\leq k\leq s$, $\lambda^k_{f}(\0)=\lambda^k_{g}(\0)$.
\end{thm}

\medskip

Given the Morse Splitting Lemma and the analytic invariance of the generic L\^e numbers, we require only one lemma to prove \thmref{thm:mainintro}:

\begin{lem}\label{lem:mainintro} Suppose that $f(w, y_1, \dots, y_n) = w^2+ g(y_1, \dots, y_n)$ for some $g\in\C\{y_1, \dots, y_n\}$. Then $\Sigma f=\{0\}\times \Sigma g$ and the generic L\^e numbers of $f$ and $g$ at the origin(s) are equal.
\end{lem}
\noindent We prove the above lemma in \secref{sec:mainproof}.

\bigskip

In \secref{sec:small}, we look at corollaries to \thmref{thm:mainintro} when $\lambda^0_f(\0)\leq 3$. In \secref{sec:smalls}, we consider the special case where $s\leq 1$; in this case, it is very easy to say what numerically generic coordinates are.

\bigskip

\section{Proof of the main lemma and theorem}\label{sec:mainproof}

As we wrote in the introduction, the only extra piece required to prove \thmref{thm:mainintro} is the proof of \lemref{lem:mainintro}. Rather than an algebraic proof, we give an elegant, quick topological proof.

\medskip

\noindent{\bf Proof}:

\smallskip

By Chapter 10 of \cite{lecycles} or   Theorem 3.6.ii of \cite{lemodtrace}, the L\^e numbers of $f$ at $\0$ with respect to generic coordinates are the ranks of the stalk cohomology in degree 0 of the iterated nearby and vanishing cycles using the generic coordinate functions applied to the vanishing cycles $\phi_f[-1]\Z^\bullet_\U[n+1]$. Since the support of $\phi_f[-1]\Z^\bullet_\U[n+1]$ is $\Sigma f$, nothing changes if we replace $\phi_f[-1]\Z^\bullet_\U[n+1]$ by the restriction
$$
\big(\phi_f[-1]\Z^\bullet_\U[n+1]\big)_{|_{\Sigma f}}.
$$

Now our Sebastiani-Thom Isomorphism in the derived category from \cite{masseysebthom} tells us that for $f(w, y_1, \dots, y_n) = w^2+ g(y_1, \dots, y_n)$, 
the perverse sheaf $\big(\phi_f[-1]\Z^\bullet_\U[n+1]\big)_{|_{\Sigma f}}$ is isomorphic to $\big(\phi_g[-1]\Z^\bullet_{\U'}[n]\big)_{|_{\Sigma g}}$, where we have identified $\Sigma f=\{0\}\times\Sigma g$ with $\Sigma g$. Therefore the ranks of the stalk cohomology in degree 0 of the iterated nearby and vanishing cycles using  generic coordinate functions applied to the vanishing cycles along $f$ and $g$ are equal.

This proves the lemma and, hence, proves the theorem.

\bigskip

\section{The cases where $\lambda^0\leq 3$}\label{sec:small}

For numerically generic coordinates $\mbf z$ for $f$ at $\0$,
$$
\lambda^0_f(\0)=\lambda^0_{f, \mbf z}(\0)=\left(\Gamma^1_{f, \mbf z}\cdot V\Big(\frac{\partial f}{\partial z_0}\Big)\right)_\0,
$$
where $\Gamma_f^1:=\Gamma^1_{f, \mbf z}$ is the (possibly non-reduced) relative polar curve of Hamm, L\^e, and Teissier as an analytic cycle. Using our notation from \cite{lecycles}, we let
$$
\gamma^1_f(\0):=\left(\Gamma^1_{f}\cdot V(z_0)\right)_\0= \mult_\0\Gamma^1_{f}.
$$

\bigskip

\noindent{\bf $\lambda^0_f(\0)=0$ case}:

\medskip

Since we are assuming that $\0\in\Sigma f$, if $\lambda^0_f(\0)=0$, then $\Gamma^1_f$ must be zero as a cycle (or empty as a set). This would, for instance, be the case if the hypersurface $V(f)$ were a product with a complex disk, i.e., $f(z_0, \dots, z_n)=g(z_1,\dots, z_n)$. More generally, by L\^e's Attaching Result in \cite{leattach}, $\lambda^0_f(\0)=0$ tells one that the Milnor fiber $F_{f,\0}$ of $f$ at $\0$ is homeomorphic to the cross product of a complex disk with the Milnor fiber $F_{f_{|_{V(z_0)}}, \0}$.

\bigskip

\noindent Now we will consider cases where $\lambda^0_f(\0)\neq0$.

\medskip

When $\Gamma_f^1\neq 0$, there is the fundamental inequality:
$$
\lambda^0_f(\0)=\left(\Gamma^1_{f, \mbf z}\cdot V\Big(\frac{\partial f}{\partial z_0}\Big)\right)_\0\geq \left(\mult_\0 \Gamma^1_{f, \mbf z}\right)\left(\mult_\0 \frac{\partial f}{\partial z_0}\right).
$$
Since $z_0$ is generic, 
$$
\mult_\0 \frac{\partial f}{\partial z_0}= -1+\mult_\0 f.
$$
Furthermore, from \cite{levan}, we know that the complex link of $V(f)$ at the origin has the homotopy-type of a wedge of  $\mult_\0 \Gamma^1_{f, \mbf z}$ $(n-1)$-spheres.

\medskip

Now, from the above discussion and induction,  the following corollaries to \thmref{thm:mainintro} are immediate.

\medskip

\begin{cor}\label{cor:1case} Suppose that $\lambda^0_f(\0)=1$. Then $s=0$ and $f$ has a complex nondegenerate critical point at $\0$, i.e., the Milnor number of $f$ at $\0$ is 1.
\end{cor}

\medskip

\begin{rem} 
In Corollary 4.5 of \cite{lemodtrace}, part of what we conclude is \corref{cor:1case} in the special case where $\Sigma f$ is smooth at $\0$. In Remark 4.3 of \cite{lemod2}, we indicated that we did not know if the generic $\lambda^0_{f, \mbf z}(\0)$ could be 1 if $\dim_\0\Sigma f=1$; the corollary tells us that this is not possible.
\end{rem}

\medskip

\begin{cor} Suppose that $\lambda^0_f(\0)=2$ (resp., $=3$) and that the rank of the hessian matrix of $f$ at $\0$ is $m$.  Then $m\leq n$ and there is an analytic change of coordinates $\eta$ at the origin in $\C^{n+1}$ and $g\in\C\{u_0, \dots, u_{n-m}\}$ such that
$$
(f\circ \eta)(w_1,\dots, w_m, u_0, \dots, u_{n-m}) = w_1^2+\dots+w_m^2+g(u_0, \dots, u_{n-m}),
$$
where 
\begin{itemize}
\item $\mult_\0 g= 3$ (resp., $=3$ or $4$), 
\medskip
\item $\gamma^1_g(\0)=1$ and so the complex link of $V(g)$ at the origin has the homotopy-type of a single $(n-m-1)$-sphere (provided $m\neq n$), 
\medskip
\item $\Sigma (f\circ\eta)=\{\0\}\times\Sigma g$, and 
\medskip
\item the generic L\^e numbers of $f$ equal the generic L\^e numbers of $g$, i.e., for all $k$ such that $0\leq k\leq s$, $\lambda^k_{f}(\0)=\lambda^k_{g}(\0)$.
\end{itemize}
\end{cor}

\bigskip

\section{1-dimensional Critical Loci}\label{sec:smalls}

If $\dim_\0\Sigma f=1$, then -- looking at the definitions in \cite{lecycles} -- one finds that the L\^e numbers $\lambda^*_{f, \mbf z}(\0)$ are defined if and only  $z_0$ is prepolar at $\0$, which is true if and only if $\dim_\0\Sigma (f_{|_{V(z_0)}})= 0$. Furthermore, for all coordinates for which the L\^e numbers are defined, the 1-dimensional L\^e cycle $\Lambda^1_{f,\mbf z}$ is independent of the coordinates; as a set, it consists simply of  $\Sigma f$ and, as a cycle, each irreducible component of $\Sigma f$ at the origin is multiplied by the Milnor number of a transverse hyperplane slice at a point near, but unequal to, the origin.

Assume $\dim_\0\Sigma f=1$  and that the L\^e numbers at $\0$ are defined. By Theorem 3.3 of \cite{lecycles}, the Euler characteristic of the Milnor fiber and the L\^e numbers are related by
$$
\chi(F_{f, \0})=1+(-1)^{n-1}\lambda^1_{f, \mbf z}(\0)+(-1)^n\lambda^0_{f, \mbf z}(\0),
$$
and so the value of $\lambda^0_{f, \mbf z}(\0)-\lambda^1_{f, \mbf z}(\0)$ is independent of the choice of coordinates for which the L\^e numbers are defined. Hence, when $\lambda^1_{f, \mbf z}(\0)$ has its minimum value, so does $\lambda^0_{f, \mbf z}(\0)$ and the coordinates $\mbf z$ would be numerically generic.

\medskip

As $\Lambda^1_{f,\mbf z}$ is independent of the coordinates (provided the L\^e numbers are defined), we conclude:

\begin{prop}\label{prop:1dim} Suppose that $\dim_\0\Sigma f=1$. Then the L\^e numbers attain their numerically generic values using the coordinates $\mbf z$  if and only if   $$\dim_\0\Sigma (f_{|_{V(z_0)}})= 0\hskip 0.1in\textnormal{ and }\hskip 0.1in \mult_\0|\Sigma f|=\big(|\Sigma f|\cdot V(z_0)\big)_\0,$$ where $|\Sigma f|$ denotes the critical locus with its reduced structure, i.e., as an analytic set.
 \end{prop}
 
 \medskip
 
 Finally, we give an example where $\dim_\0\Sigma f=1$, the coordinates $\mbf z$ are prepolar (so the L\^e numbers are defined), and $\lambda^0_{f, \mbf z}(\0)=1$.  By \corref{cor:1case}, we know that such $\mbf z$ cannot be numerically generic.
 
 \begin{exm} Let $f(x,y,z)= z^2+(x-y^2)^2$. Then,  $\Sigma f= V(x-y^2, z)$. Let $\mbf u$ be the coordinates $(x,y,z)$ and let $\mbf v$ be the coordinates $(y,x,z)$. By \propref{prop:1dim}, the coordinates $\mbf u$ are not numerically generic, but the coordinates $\mbf v$ are.
 
 The reader is invited to calculate:
$$\lambda^0_{f, \mbf u}(\0)=1, \ \lambda^1_{f, \mbf u}(\0)=2, \hskip 0.1in \textnormal{ and }\hskip 0.1in \lambda^0_{f, \mbf v}(\0)=0, \ \lambda^1_{f, \mbf v}(\0)=1.$$
 \end{exm}


\bigskip
 
\begin{thebibliography}{10}

\bibitem{GLS}
{Greuel, G.-M., Lossen, C., and Shustin, E.}
\newblock {\em {Introduction to Singularities and Deformations}}.
\newblock {Springer Monographs in Math.} {Springer}, 2025.

\bibitem{leattach}
{L\^e, D. T.}
\newblock {Calcul du Nombre de Cycles \'Evanouissants d'une Hypersurface
  Complexe}.
\newblock {\em Ann. Inst. Fourier, Grenoble}, 23:261--270, 1973.

\bibitem{levan}
{L\^e, D. T.}
\newblock {Sur les cycles \'evanouissants des espaces analytiques}.
\newblock {\em C. R. Acad. Sci. Paris, S\'er. A-B}, 288:A283--A285, 1979.

\bibitem{lemassey}
{Le, D. T., Massey, D.}
\newblock {Hypersurface Singularities and Milnor Equisingularity}.
\newblock {\em Pure and Appl. Math. Quart., special issue in honor of Robert
  MacPherson's 60th birthday}, 2, no.3:893--914, 2006.

\bibitem{levar1}
{Massey, D.}
\newblock {The L\^e Varieties, I}.
\newblock {\em Invent. Math.}, 99:357--376, 1990.

\bibitem{levar2}
{Massey, D.}
\newblock {The L\^e Varieties, II}.
\newblock {\em Invent. Math.}, 104:113--148, 1991.

\bibitem{lecycles}
{Massey, D.}
\newblock {\em {L\^e Cycles and Hypersurface Singularities}}, volume 1615 of
  {\em Lecture Notes in Math.}
\newblock Springer-Verlag, 1995.

\bibitem{masseysebthom}
{Massey, D.}
\newblock {The Sebastiani-Thom Isomorphism in the Derived Category}.
\newblock {\em Compos. Math.}, 125:353--362, 2001.

\bibitem{lemodtrace}
{Massey, D.}
\newblock {L\^e Modules and Traces}.
\newblock {\em Proc. AMS}, 134 (7):2049--2060, 2006.

\bibitem{lemod2}
{Massey, D.}
\newblock Lê modules and hypersurfaces with one-dimensional singular sets.
\newblock {\em Acta Math. Vietnamica}, 2026.
\newblock to appear.

\end{thebibliography}
\end{document}